\documentclass[final]{siamltex}

\usepackage{epsf, graphicx}
\usepackage{latexsym,amsfonts,amsbsy,amssymb}
\usepackage{amsmath}
\usepackage{float}
\usepackage{color}
\usepackage{hyperref,color}
\allowdisplaybreaks 

\allowdisplaybreaks 
\newtheorem{remark}{Remark}[section]

\title{Sharp Error Bounds for Piecewise Polynomial Approximation: Revisit and Application to Elliptic
PDE Eigenvalue Computation}
\author{Hehu Xie\thanks{LSEC,
Institute of Computational Mathematics, Academy of Mathematics and
Systems Science, Chinese Academy of Sciences, Beijing 100190,
China (hhxie@lsec.cc.ac.cn).}
\and Zhimin Zhang\thanks{Beijing Computational Science Research Center, and
Department of Mathematics, Wayne State University, Detroit, MI 48202, USA
(zzhang@math.wayne.edu).}}
\begin{document}
\maketitle
\begin{abstract}
In this paper, we revisit approximation properties of piecewise polynomial spaces,
which contain more than ${\cal P}_{r-1}$ but not ${\cal P}_r$.
We develop more accurate upper and lower error bounds that are sharper than those used in literature.
These new error bounds, especially the lower bounds are particular useful to finite element methods.
As an important application, we establish sharp lower bounds of the discretization error
for Laplace and $2m$-th order elliptic eigenvalue problems in various finite element spaces
under shape regular triangulations, and investigate the asymptotic convergence behavior for large
numerical eigenvalue approximations.
\end{abstract}
\pagestyle{myheadings}
\thispagestyle{plain}
\markboth{Hehu Xie and Zhimin Zhang}{Sharp Error Bounds for Polynomial Approximation}

\section{Introduction}
Standard error estimates for numerical methods are conventionally presented as
upper bounds. For example, the error estimate for the finite element method
 is usually written in the following way
 \begin{eqnarray}\label{Upper_Bound_Normal}
 \|u-u_h\|_1 &\leq &Ch^{r-1}\|u\|_r,
 \end{eqnarray}
 where $u_h$ denotes the finite element approximation (of polynomial degree less than $r$) to the exact solution $u$.
 However, this error estimate is not so delicate and sharp for some finite elements such as the tensor product and
 intermediate families. The first aim of this paper is to establish some more delicate and sharp
 upper bounds for piecewise polynomial approximation errors on tensor product meshes by giving
   a new norm equivalence theorem.

The analysis of lower bound error estimates have been studied in the literature
\cite{BabuskaMiller,BabuskaStroulis,KrizekRoosChen,LinXieXu,Widlund1971,Widlund}.
 However, none of them are particularly for eigenvalue approximation.
In particular, when eigenvalue problems are involved,
these upper bounds are tied with the magnitudes of the underlying eigenvalues.
For the same finite element method, larger eigenvalues always have larger error bounds,  the fact was
described in Strang-Fix's 1973 classical book \cite{StrangFix}.
 Inspired by some recent studies on efficiency of finite element approximation for different eigenvalues \cite{Zhang},
this work is also devoted to establishing some lower bound error estimates for finite element
 approximation of eigenvalue problems.
Then, the second aim of this paper is to derive lower error bounds
for eigenvalue approximations by finite element methods with continuous or discontinuous
 piecewise polynomial functions under reasonable assumptions on regularities of eigenfunctions and triangulations.
 As an application, we show the asymptotic convergence behavior for the large numerical eigenvalue approximations.
This knowledge can help us to estimate the number of reliable numerical eigenvalues that convergence at a predetermined
rate, and thereby provide a solid theoretical foundation for the issue raised and discussed in \cite{Zhang}.

 Towards this end, we revisited piecewise polynomial approximations and obtain more delicate error estimates
 in terms of higher-order derivatives.
 Then combining { norm estimates of an eigenfunction in terms of the underlying
 eigenvalue leads to a lower bound of the approximation error for the eigenvalue problem in the following form
(see Sections 3 and 4):}
\begin{equation*}
\|u-u_h\|_{j,p,h} \geq  C \lambda^{\frac{r}{2}} h^{r-j},\ \ \ \ 0\leq j\leq r,
\end{equation*}
where $u_h$ is the finite element approximation (of polynomial degree less than $r$) to the Laplace eigenvalue problem,
and $C$ is a constant independent of the mesh size $h$ and the eigenvalue $\lambda$.

The outline of the rest of the paper goes as follows. Section 2 is
devoted to introducing the notation and a new norm equivalence theorem.
Then some upper bounds of the error by
piecewise polynomial approximation are given in Section 3.
Section 4 is for lower bounds of
the discretization error of the Laplace eigenvalue problem
by finite element methods. The lower bounds of the discretization error
 of the $2m$-th order elliptic eigenvalue problem are presented in Section 5.
 Some concluding remarks are given in the last section.

\section{A norm equivalent theorem: Revisit}
In this section, we first introduce some notations, and then state a new norm
equivalence theorem, from which we derive new upper bounds of the error by piecewise polynomial approximation.

Here we assume that $\Omega\subset \mathcal{R}^n$ ($n\geq 1$) is a bounded polytopic domain with Lipschitz
continuous boundary $\partial\Omega$.  Throughout this paper, we use the standard notation for the usual
Sobolev spaces and the corresponding norms, semi-norms, and inner products as in
 \cite{Adams,BrennerScott, Ciarlet}.

A multi-index $\alpha$ is an $n$-tuple of non-negative
integers $\alpha_i$ with length $\displaystyle |\alpha|=\sum_{i=1}^n\alpha_i$.
The derivative $D^{\alpha}v$ is then defined by
$$D^{\alpha}v=\Big(\frac{\partial}{\partial x_1}\Big)^{\alpha_1}\cdots \Big(\frac{\partial}{\partial x_n}\Big)^{\alpha_n}v.$$
For a subdomain $G$ of $\Omega$, the usual Sobolev spaces $W^{m,p}(G)$  with norm $\|\cdot\|_{m,p,G}$
and semi-norm $|\cdot|_{m,p,G}$ are used. In the case $p=2$, we have $H^m(G)=W^{m,2}(G)$ and the index
$p$ will be omitted. The $L^2$-inner product on $G$ is denoted by $(\cdot,\cdot)_G$. For $G\subset \Omega$
 we write $G\subset\subset \Omega$ to indicate that ${\rm dist}(\partial\Omega, G)>0$ and ${\rm meas}(G)>0$.

We introduce a face-to-face partition $\mathcal{T}_h$ of the computational domain $\Omega$ into
elements $K$ (triangles, rectangles, tetrahedrons, bricks, etc.) such that
$$\bar{\Omega}=\bigcup_{K\in\mathcal{T}_h}K$$
and let $\mathcal{E}_h$ denote  a set of all $(n-1)$-dimensional facets of all elements $K\in\mathcal{T}_h$.
Here $h:=\max_{K\in\mathcal{T}_h} h_K$ and $h_K=\text{diam}\ K$
denote the global and local mesh size, respectively \cite{BrennerScott,Ciarlet}.
We also define $\mathcal{T}_h^G=\big\{K\in \mathcal{T}_h\ {\rm and}\ K\subset G\big\}$ and
$h_G=\max_{K\in \mathcal{T}_h^G} h_K$.
A family of partitions $\mathcal{T}_h$ is said to be {\it shape regular} if it satisfies the following
condition \cite{BrennerScott}:
$$\exists \sigma>0 \makebox{ such that} \ \ {h_K}/{\tau_K}>\sigma \ \ \ \forall K\in \mathcal{T}_h, $$
where $\tau_K$ is maximum diameter of the inscribed ball in $K\in\mathcal{T}_h$. A { shape} regular family of
partitions $\mathcal{T}_h$ is called {\it quasi-uniform} if it satisfies \cite{BrennerScott,Ciarlet}
$$\exists \beta>0\ \ \makebox{
such that}\ \ \max\{{h}/{h_K}, \ K\in \mathcal{T}_h\}\leq \beta.$$
A finite element space $V_h$ of piecewise polynomial functions is constructed on $\mathcal{T}_h$ as
\begin{eqnarray}\label{FEM_Space_General}
V_h=\big\{v : v|_K\in\mathcal{P}_{\rm used},\ \ \forall K\in \mathcal{T}_h\big\},
\end{eqnarray}
where $\mathcal{P}_{\rm used}$ denotes the polynomial space associated with
multi-index set ${\rm Ind}_{\rm used}$
\begin{eqnarray*}
\mathcal{P}_{\rm used}=\Big\{v: v=\sum_{\alpha\in{\rm Ind}_{\rm used}}C_{\alpha}x^{\alpha}\Big\}.
\end{eqnarray*}
We define multi-index sets
$$
{\rm Ind}_r=\big\{\alpha:\; |\alpha|\leq r \big\},
\quad {\rm Ind}_{r, \rm rest}={\rm Ind}_{r}\backslash{\rm Ind}_{\rm used},
$$
and assume that ${\rm Ind}_{r-1}\subseteq {\rm Ind}_{\rm used}$ and
${\rm Ind}_r\nsubseteqq{\rm Ind}_{\rm used}$.
We further define the following broken semi-norm for $v\in W^{j,p}(G)\cup V_h$ with $G\subseteq \Omega$
\begin{eqnarray*}
|v|_{j,p,G,h}&:=&\left(\sum_{K\in\mathcal{T}_h^G}\int_{K}\sum_{|\alpha|=j}|D^{\alpha}v|^pdK\right)^{\frac{1}{p}},
\ \ \ \ 1\leq p<\infty,
\end{eqnarray*}
and
\begin{eqnarray*}
|v|_{j,\infty,G,h}&:=&\max_{K\in \mathcal{T}_h^G}|v|_{j,\infty,K}.
\end{eqnarray*}
Then the corresponding norm can be defined by
\begin{eqnarray*}
\|v\|_{j,p,G,h}&:=&\left(\sum_{i=0}^j|v|_{i,p,G,h}^p\right)^{\frac{1}{p}},
\end{eqnarray*}
and
\begin{eqnarray*}
\|v\|_{j,\infty,G,h}&:=&\max_{0\leq i\leq j}|v|_{i,\infty,G,h}.
\end{eqnarray*}
We will drop $G$ when $G=\Omega$. Throughout this paper, the symbol
$C$ or $c$ (with or without subscript) stands for a positive generic constant which may attain different
values at its different occurrences and which is independent of the mesh size $h$.

Now, we state a more general norm equivalence theorem{, of which the conventional} norm equivalence theorem
is a special case.
\begin{theorem}\label{Equivalence_Norm_Theorem}
Let $\Omega\subset \mathcal{R}^n$ be a Lipschitz domain{. For $j=1,2,\ldots,J$,
Assume that $f_j: W^{r+1,p}(\Omega)\rightarrow \mathcal{R}$ with $r\geq 1$ and $1\leq p <\infty$} are
semi norms on $W^{r+1,p}(\Omega)$ satisfying two conditions:

(H1)\ $0\leq f_j(v)\leq C\|v\|_{r+1,p,\Omega},\  \forall v\in W^{r,p}(\Omega)$, $1\leq j\leq J$.

(H2)\ If $v\in \mathcal{P}_{\rm used}$  and $f_j(v)=0$ for $1\leq j\leq J$, then $v=0$.\\
Then the quantity
\begin{eqnarray}\label{Equiv_Norm_1}
\|v\|=\sum_{\gamma\in {\rm Ind}_{r,\rm rest}}\|D^{\gamma}v\|_{0,p,\Omega} +\sum_{j=1}^Jf_j(v) +
\sum_{|\gamma|=r+1}\|D^{\gamma}v\|_{0,p,\Omega}
\end{eqnarray}
or
\begin{eqnarray}\label{Equiv_Norm_2}
\|v\|=\left(\sum_{\gamma\in {\rm Ind}_{r,\rm rest}}\|D^{\gamma}v\|_{0,p,\Omega}^p +\sum_{j=1}^Jf_j(v)^p
+\sum_{|\gamma|=r+1}\|D^{\gamma}v\|_{0,p,\Omega}^p\right)^{\frac{1}{p}}
\end{eqnarray}
defines a norm on $W^{r+1,p}(\Omega)$, which is equivalent to the norm $\|v\|_{r+1,p,\Omega}$.
\end{theorem}
\begin{proof}
We prove that the quantity (\ref{Equiv_Norm_1}) defines a norm on $W^{r+1,p}(\Omega)$
equivalent to the norm $\|v\|_{r+1,p,\Omega}$. The statement on the quantity (\ref{Equiv_Norm_2})
can be proved similarly or by noting the equivalence between the two quantities (\ref{Equiv_Norm_1})
and (\ref{Equiv_Norm_2}).

By the condition { $(H1)$}, we know that for some constant $C>0$, the following estimate holds
\begin{eqnarray}\label{Inequality_1}
\|v\|\leq C\|v\|_{r+1,p,\Omega}\ \ \ \ \ \forall v\in W^{r+1,p}(\Omega).
\end{eqnarray}
So we only need to prove that there exist another constant $C>0$ such that
\begin{eqnarray}\label{Inequality_2}
\|v\|_{r+1,p,\Omega}\leq C\|v\|\ \ \ \ \ \forall v\in W^{r+1,p}(\Omega).
\end{eqnarray}
We argue by contradiction. Suppose the inequality is not true. Then we can find a sequence
$\{v_k\}\subset W^{r+1,p}(\Omega)$ with the properties
\begin{eqnarray}
\|v_k\|_{r+1,p,\Omega}&=&1,\label{Assum_1}\\
\|v_k\|&\leq&\frac{1}{k}\label{Assum_2}
\end{eqnarray}
for $k=1,2,\cdots$. From (\ref{Assum_2}), we can see that as $k\rightarrow \infty$,
\begin{eqnarray}\label{Condition_1}
\sum_{\gamma\in {\rm Ind}_{r,\rm rest}}\|D^{\gamma}v_k\|_{0,p,\Omega} \rightarrow 0,
\end{eqnarray}
\begin{eqnarray}\label{Condition_2}
\sum_{|\gamma|=r+1}\|D^{\gamma}v_k\|_{0,p,\Omega} \rightarrow 0
\end{eqnarray}
and
\begin{eqnarray}\label{Condition_3}
f_j(v_k)\rightarrow0,\ \ \ \ \ 1\leq j\leq J.
\end{eqnarray}
Since $\{v_k\}$ is a bounded sequence on $W^{r+1,p}(\Omega)$ from
the property (\ref{Assum_1}) and $W^{r+1,p}(\Omega)$ can be embedded compactly into
$W^{r,p}(\Omega)$, there is a subsequence of the sequence $\{v_k\}$, still denoted
as $\{v_k\}$, and a function $v\in W^{r,p}(\Omega)$ such that
\begin{eqnarray}
v_k\rightarrow v\ \ \ \ \ {\rm in}\ W^{r,p}(\Omega),\ \ \ {\rm as}\ k\rightarrow \infty.
\end{eqnarray}
This property and (\ref{Condition_1})-(\ref{Condition_2}), together with the uniqueness of a limit, imply that
\begin{eqnarray*}
v_k\rightarrow v\ \ \ {\rm in\ the \ norm}\ \|v\|_{r+1,p,\Omega},
\end{eqnarray*}
as ${ k}\rightarrow \infty$, and
\begin{eqnarray*}
&&\sum_{\gamma\in{\rm Ind}_{r,\rm rest}}\|D^{\gamma}v\|_{0,p,\Omega}+\sum_{|\gamma|=r+1}\|D^{\gamma}v\|_{0,p,\Omega}\nonumber\\
&=&\lim_{k\rightarrow \infty}\left(\sum_{\gamma\in{\rm Ind}_{r,\rm rest}}\|D^{\gamma}v_k\|_{0,p,\Omega}
+\sum_{|\gamma|=r+1}\|D^{\gamma}v_k\|_{0,p,\Omega}\right) =0.
\end{eqnarray*}
We then conclude that $v\in\mathcal{P}_{\rm used}(\Omega)$. On the other hand, from
 continuity of functionals $\{f_j\}_{1\leq j\leq J}$ and (\ref{Assum_2}), we find that
\begin{eqnarray*}
f_j(v)=\lim_{k\rightarrow \infty}f_j(v_k)=0.
\end{eqnarray*}
Using the condition { $(H2)$}, we see that $v=0$, which contradicts the relation that
\begin{eqnarray*}
\|v\|_{r+1,p,\Omega} =\lim_{k\rightarrow \infty}\|v_k\|_{r+1,p,\Omega}=1.
\end{eqnarray*}
Hence the inequality (\ref{Inequality_2}) holds which means the norm $\|\cdot\|$ is equivalent
to the one $\|\cdot\|_{r+1,p,\Omega}$ and the proof is completed.
\end{proof}

In the error analysis for the finite element method, we also need an inequality involving
the norm of the Sobolev quotient space
\begin{eqnarray}
V=W^{r+1,p}(\Omega)/\mathcal{P}_{r,\rm used} =\Big\{[v]: [v]=\{v+q| q\in\mathcal{P}_{r,\rm used}\}\ \ \forall
v\in W^{r+1,p}(\Omega)\Big\},
\end{eqnarray}
where $\mathcal{P}_{r,\rm used}$ denotes the polynomial space associated with
multi-index set ${\rm Ind}_{r,\rm used}={\rm Ind}_r\cap {\rm Ind}_{\rm used}$.

Any element $[v]$ of the space $V$ is an equivalence class, the difference between any two elements in
the equivalence class being a polynomial in the space $\mathcal{P}_{r,\rm used}$.
Any $[v]$ is called a representative element $[v]$. The quotient norm in the space $V$ is defined to be
\begin{eqnarray*}
\|[v]\|_V=\inf_{q\in\mathcal{P}_{r,\rm used}(\Omega)}\|v+q\|_{r+1,p,\Omega}.
\end{eqnarray*}
\begin{corollary}\label{Quotient_Norm_Corollary}
Let $\Omega\subset \mathcal{R}^n$ be a Lipschitz domain. Then
the quantity \\ $\sum_{\gamma\in{\rm Ind}_{r,\rm rest}}\|D^{\gamma}v\|_{0,p,\Omega}$ $+
\sum_{|\gamma|=r+1}\|D^{\gamma}v\|_{0,p,\Omega}$ {\rm ($1\leq p<\infty$)}
is a semi-norm on $V$, equivalent to the quotient norm $\|[v]\|_V$. It means
\begin{eqnarray}
&&C_1\left(\sum_{\gamma\in{\rm Ind}_{r,\rm rest}}\|D^{\gamma}v\|_{0,p,\Omega}+\sum_{|\gamma|=r+1}\|D^{\gamma}v\|_{0,p,\Omega}\right) \leq
\inf_{q\in\mathcal{P}_{r,\rm used}(\Omega)}\|v+q\|_{r+1,p,\Omega}\nonumber\\
 &&\ \ \ \ \ \ \ \ \ \ \leq C_2 \left(\sum_{\gamma\in{\rm Ind}_{r,\rm rest}}\|D^{\gamma}v\|_{0,p,\Omega}+\sum_{|\gamma|=r+1}\|D^{\gamma}v\|_{0,p,\Omega}\right),
\end{eqnarray}
where $C_1$ and $C_2$ are two constants independent of $v$.
\end{corollary}
\begin{proof}
Obviously, for any $[v]\in V$ and any $v\in [v]$, we have
\begin{eqnarray*}
\|[v]\|_V=\inf_{q\in\mathcal{P}_{r,\rm used}(\Omega)}\|v+q\|_{r+1,p,\Omega}\geq \sum_{\gamma\in{\rm Ind}_{r,\rm rest}}\|D^{\gamma}v\|_{0,p,\Omega}+\sum_{|\gamma|=r+1}\|D^{\gamma}v\|_{0,p,\Omega}.
\end{eqnarray*}
Thus we only need to prove that there is a constant $C$, depending only on $\Omega$, such that
\begin{eqnarray}\label{Quotient_Norm_Inequality}
\inf_{q\in\mathcal{P}_{r,\rm used}(\Omega)}\|v+q\|_{r+1,p,\Omega}\leq
C\left(\sum_{\gamma\in{\rm Ind}_{r,\rm rest}}\|D^{\gamma}v\|_{0,p,\Omega}+\sum_{|\gamma|=r+1}\|D^{\gamma}v\|_{0,p,\Omega}\right).
\end{eqnarray}
Denote $N = {\rm dim}(\mathcal{P}_{r,\rm used}(\Omega))$.
Define $N$ independent linear continuous functionals
on $\mathcal{P}_{r,\rm used}$,  the continuity being with respect to the norm of $W^{r+1,p}(\Omega)$.
Now we can use Theorem \ref{Equivalence_Norm_Theorem} to obtain the following inequality
\begin{eqnarray}
\|v\|_{r+1,p,\Omega}&\leq& C\left(\sum_{\gamma\in{\rm Ind}_{r,\rm rest}}\|D^{\gamma}v\|_{0,p,\Omega}
+\sum_{|\gamma|=r+1}\|D^{\gamma}v\|_{0,p,\Omega}\right.\nonumber\\
&&\ \ \ \ \ \ \ \
\left.+\sum_{\gamma\in{\rm Ind}_{r,\rm used}}\Big|\int_{\Omega}D^{\gamma}vd\Omega\Big|\right), \ \ \forall v\in W^{r+1,p}(\Omega).
\end{eqnarray}
Replacing $v$ by $v+q$ and noting that $D^{\gamma}q=0$ for
$\gamma\in{\rm Ind}_{r,\rm rest}\cup {\rm Ind}_{r+1}$, we have
\begin{eqnarray}\label{Quotient_1}
\|v+q\|_{r+1,p,\Omega}&\leq&
C\left(\sum_{\gamma\in{\rm Ind}_{r,\rm rest}}\|D^{\gamma}v\|_{0,p,\Omega}
+\sum_{|\gamma|=r+1}\|D^{\gamma}v\|_{0,p,\Omega}\right.\nonumber\\
&&\hskip-3cm\left.+\sum_{\gamma\in{\rm Ind}_{r,\rm used}}
\Big|\int_{\Omega}D^{\gamma}(v+q)d\Omega\Big|\right),
\ \forall v\in W^{r+1,p}(\Omega)\ {\rm and}\ \forall q\in \mathcal{P}_{r,\rm used}.
\end{eqnarray}
Now we construct a polynomial $\bar{q}\in \mathcal{P}_{r,\rm used}$ satisfying
\begin{eqnarray}\label{Quotient_2}
\int_{\Omega}D^{\gamma}(v+\bar{q})d\Omega=0,\ \ \ \ {\rm for}\ \gamma\in {\rm Ind}_{r,\rm used}.
\end{eqnarray}
This can always be done: set $\gamma\in {\rm Ind}_{r,\rm used}$, then $D^{\gamma}\bar{q}$ equals
$\gamma!=\gamma_1!\cdots\gamma_n!$ times the coefficient of $x^{\gamma}:=x_1^{\gamma_1}\cdots x_n^{\gamma_n}$,
so the coefficient can be computed by
using (\ref{Quotient_2}). Having found all the coefficients for terms in $\mathcal{P}_{r,\rm used}$,
we set $\gamma\in{\rm Ind}_{r,\rm used}$, and use (\ref{Quotient_2}) to compute all the coefficients for terms of
$\mathcal{P}_{r,\rm used}$. Proceeding in this way, we obtain the polynomial $\bar{q}$ satisfying
the condition (\ref{Quotient_2}) for the given function $v$.

With $q =\bar{q}$ in (\ref{Quotient_1}), we have
\begin{eqnarray*}
\inf_{q\in\mathcal{P}_{r,\rm used}(\Omega)}\|v+q\|_{r+1,p,\Omega}&\leq& \|v+\bar{q}\|_{r+1,p,\Omega}\nonumber\\
&\leq& C\left(\sum_{\gamma\in{\rm Ind}_{r,\rm rest}}\|D^{\gamma}v\|_{0,p,\Omega}
+\sum_{|\gamma|=r+1}\|D^{\gamma}v\|_{0,p,\Omega}\right).
\end{eqnarray*}
This is the desired result (\ref{Quotient_Norm_Inequality}) and the proof is complete.
\end{proof}

Corollary \ref{Quotient_Norm_Corollary} is more delicate than the { standard} norm equivalence theorem.
In some special meshes, this new estimate can lead to new and sharper error estimates for the
piecewise polynomial approximation.

\section{New upper error bounds for piecewise polynomial approximation}
In this section, as applications of the new  norm equivalence theorem stated in Theorem \ref{Equivalence_Norm_Theorem},
we give some error estimates for the tensor product and intermediate finite element spaces.
\begin{theorem}\label{Upper_Bound_Theorem}
Assume $u\in W^{r+1,p}(\Omega)$ and the mesh is tensor-product type with the mesh size
$h_{K,1},\cdots,h_{K,n}$ in each direction and the edges paralleling to the axes
for each $K\in\mathcal{T}_h$. We use the finite element space $V_h$
defined in (\ref{FEM_Space_General}) to approximate the function $u$.
Then the following upper bound of the error holds
\begin{eqnarray}\label{Upper_Bound_Error}
\inf_{v_h\in V_h}\|D^{\alpha}(u-v_h)\|_{0,p,\Omega,h} &\leq& C
\left( \sum_{K\in\mathcal{T}_h}\sum_{\gamma\in {\rm Ind}_{r,\rm rest}}
h_K^{p(\gamma-\alpha)} \|D^{\gamma}u\|_{0,p,K}^p\right.\nonumber\\
&&\ \ \ \ \ \left.+\sum_{K\in\mathcal{T}_h}\sum_{|\gamma|=r+1}
h_K^{p(\gamma-\alpha)} \|D^{\gamma}u\|_{0,p,K}^p \right)^{\frac{1}{p}},
\end{eqnarray}
where $h_K^{p(\gamma-\alpha)}=h_{K,1}^{p(\gamma_1-\alpha_1)}\cdots h_{K,n}^{p(\gamma_n-\alpha_n)}$.
\end{theorem}
\begin{proof}
The proof can be given by combing Corollary \ref{Quotient_Norm_Corollary},
${\rm Ind}_{r,\rm used}\subseteq {\rm Ind}_{\rm used}$  and the standard
scaling argument (cf. \cite{BrennerScott,Ciarlet}).
\end{proof}

Theorem \ref{Upper_Bound_Theorem} can be used to derive { sharper (than conventional) error estimates
for tensor product and intermediate
finite element spaces} \cite{BabuskaStroulis,BrennerScott,Ciarlet}.
 The tensor product finite element space on tensor product meshes is
 defined as follows \cite{BabuskaStroulis,BrennerScott,Ciarlet}:
\begin{eqnarray}\label{Tensor_Product}
V_h&=&\Big\{v_h:\ v_h|_K\in\mathcal{Q}_{r-1},\ \ \forall K\in\mathcal{T}_h\Big\}.
\end{eqnarray}
Here and hereafter $\mathcal{Q}_{k}$ denotes the space of polynomials with a degree no more than $k$ for each
variable, i.e. $\mathcal{Q}_k=\{x^{\alpha}$ with $\alpha_i \leq k$ for $i=1,\cdots,n\}$.
For example, when $V_h$ is the biquadratic element with $r=3$ and $n=2$, we know
${\rm Ind}_{3,{\rm rest}}=\{(3,0), (0,3)\}$ and we have the following error estimates
\begin{eqnarray}
\inf_{v_h\in V_h}\|u-v_h\|_0 &\leq& C\left(\sum_{K\in\mathcal{T}_h}h_{K,1}^{6}\|\partial_{xxx}u\|_0^2+
h_{K,2}^{6}\|\partial_{yyy}u\|_0^2\right)^{\frac{1}{2}}\nonumber\\
&&\ \ \ \ +C\left(\sum_{K\in\mathcal{T}_h}\sum_{|\gamma|=4}h_K^{2\gamma}
\|D^{\gamma}u\|_{0,2}^2\right)^{\frac{1}{2}},\\
\inf_{v_h\in V_h}\|\partial_{x}(u-v_h)\|_0 &\leq& C\left(\sum_{K\in\mathcal{T}_h}h_{K,1}^2\|\partial_{xxx}u\|_0^{4}+
h_{K,1}^{-2}h_{K,2}^{6}\|\partial_{yyy}u\|_0^2\right)^{\frac{1}{2}}\nonumber\\
&&\ \ \ \ +C\left(\sum_{K\in\mathcal{T}_h}\sum_{|\gamma|=4}h_K^{2(\gamma-(1,0))}\|D^{\gamma}u\|_{0,2}^2 \right)^{\frac{1}{2}},\\
\inf_{v_h\in V_h}\|\partial_{y}(u-v_h)\|_0 &\leq& C\left(\sum_{K\in\mathcal{T}_h}h_{K,1}^{6}h_{K,2}^{-2}\|\partial_{xxx}u\|_0^2+
h_{K,2}^{4}\|\partial_{yyy}u\|_0^2\right)^{\frac{1}{2}}\nonumber\\
&&\ \ \ \ +C\left(\sum_{K\in\mathcal{T}_h}\sum_{|\gamma|=4}
h_K^{2(\gamma-(0,1))}\|D^{\gamma}u\|_{0,2}^2\right)^{\frac{1}{2}},
\end{eqnarray}
which are different from the {traditional} error estimates.

The intermediate family of the second type on rectangular meshes
is defined as follows \cite{BabuskaStroulis,ZhangYanSun}:
\begin{eqnarray}\label{Intermediate_Type}
V_h&=&\Big\{v_h:\ v_h|_K\in\mathcal{P}_{r+1}\cap\mathcal{Q}_{r-1},\ \forall K\in\mathcal{T}_h\Big\},
\end{eqnarray}
where $\mathcal{P}_k$ denotes the space of polynomials with degree no more than $k$,
i.e. $\mathcal{P}_k=\{x^{\alpha}$ with $|\alpha|\leq k\}$. For example, when $r=4$, we have
${\rm Ind}_{4,{\rm rest}}=\{(4,0),(0,4)\}$ and the following error estimates
\begin{eqnarray}
\inf_{v_h\in V_h}\|u-v_h\|_0 &\leq& C\left(\sum_{K\in\mathcal{T}_h}h_{K,1}^{8}\|\partial_{xxxx}u\|_0^2+
h_{K,2}^{8}\|\partial_{yyyy}u\|_0^2\right)^{\frac{1}{2}}\nonumber\\
&&\ \ \ \ +C\left(\sum_{K\in\mathcal{T}_h}\sum_{|\gamma|=5}
h_K^{2\gamma}\|D^{\gamma}u\|_{0,2}^2\right)^{\frac{1}{2}},\\
\inf_{v_h\in V_h}\|\partial_{x}(u-v_h)\|_0 &\leq& C\left(\sum_{K\in\mathcal{T}_h}h_{K,1}^{6}\|\partial_{xxxx}u\|_0^2+
h_{K,1}^{-2}h_{K,2}^{8}\|\partial_{yyyy}u\|_0^2\right)^{\frac{1}{2}}\nonumber\\
&&\ \ \ \ +C\left(\sum_{K\in\mathcal{T}_h}\sum_{|\gamma|=5}
h_K^{2(\gamma-(1,0))}\|D^{\gamma}u\|_{0,2}^2\right)^{\frac{1}{2}},\\
\inf_{v_h\in V_h}\|\partial_{y}(u-v_h)\|_0 &\leq& C\left(\sum_{K\in\mathcal{T}_h}h_{K,1}^{8}h_{K,2}^{-2}\|\partial_{xxxx}u\|_0^2+
h_{K,2}^{6}\|\partial_{yyyy}u\|_0^2\right)^{\frac{1}{2}}\nonumber\\
&&\ \ \ \ +C\left(\sum_{K\in\mathcal{T}_h}\sum_{|\gamma|=5}
h_K^{2(\gamma-(0,1))}\|D^{\gamma}u\|_{0,2}^2\right)^{\frac{1}{2}}.
\end{eqnarray}

Theorem \ref{Upper_Bound_Theorem} can also be used to give a new upper bound error estimate for the
serendipity {finite element families} \cite{ArnoldAwanou,BabuskaStroulis,BrennerScott}:
\begin{eqnarray}\label{Serendipity_Type}
V_h&=&\Big\{v_h:\ v_h|_K\in\mathcal{S}_{r},\ \forall K\in\mathcal{T}_h\Big\},
\end{eqnarray}
where $\mathcal{S}_r=\mathcal{P}_{r-1}+{\rm span}\{x^{r-1}y,xy^{r-1}\}$.
For example, when $r=3$, we know
${\rm Ind}_{3,{\rm rest}}=\{(3,0), (0,3)\}$ and the following error estimates hold
\begin{eqnarray}
\inf_{v_h\in V_h}\|u-v_h\|_0 &\leq& C\left(\sum_{K\in\mathcal{T}_h}h_{K,1}^{6}\|\partial_{xxx}u\|_0^2+
h_{K,2}^{6}\|\partial_{yyy}u\|_0^2\right)^{\frac{1}{2}}\nonumber\\
&&\ \ \ \ +C\left(\sum_{K\in\mathcal{T}_h}\sum_{|\gamma|=4}
h_K^{2\gamma}\|D^{\gamma}u\|_{0,2}^2\right)^{\frac{1}{2}},\\
\inf_{v_h\in V_h}\|\partial_{x}(u-v_h)\|_0 &\leq& C\left(\sum_{K\in\mathcal{T}_h}h_{K,1}^{4}\|\partial_{xxx}u\|_0^2+
h_{K,1}^{-2}h_{K,2}^{6}\|\partial_{yyy}u\|_0^2\right)^{\frac{1}{2}}\nonumber\\
&&\ \ \ \ +C\left(\sum_{K\in\mathcal{T}_h}\sum_{|\gamma|=4}
h_K^{2(\gamma-(1,0))}\|D^{\gamma}u\|_{0,2}^2\right)^{\frac{1}{2}},\\
\inf_{v_h\in V_h}\|\partial_{y}(u-v_h)\|_0 &\leq& C\left(\sum_{K\in\mathcal{T}_h}h_{K,1}^{6}h_{K,2}^{-2}\|\partial_{xxx}u\|_0^2+
h_{K,2}^{4}\|\partial_{yyy}u\|_0^2\right)^{\frac{1}{2}}\nonumber\\
&&\ \ \ \ +C\left(\sum_{K\in\mathcal{T}_h}\sum_{|\gamma|=4}
h_K^{2(\gamma-(0,1))}\|D^{\gamma}u\|_{0,2}^2\right)^{\frac{1}{2}},
\end{eqnarray}
which are different from the {standard} error estimates.

Compared with the {traditional} upper bound (\ref{Upper_Bound_Normal}), the above
new bounds are sharper  in some sense. Theorem \ref{Upper_Bound_Theorem}
 can also been used to deduce some more delicate and sharper error estimates for other finite element families.
To end this section, we would like to point out that this type of error estimates (with some difference)
have also been given and proved  in \cite[Chapter 4.5]{BrennerScott} {by} the average Taylor polynomials technique.
Nevertheless, we need more delicate theory in our analysis and the approach based on the quotient space argument
 in this paper serves the purpose and requires minimum regularity assumption.

\section{Lower bounds for finite element approximation of Laplace eigenvalue problem}
\label{Lower_Bound_Elliptic_Eigenvalue_Section}
Another important issue in this paper is to derive the lower
bounds of the discretization error for the Laplace eigenvalue
problem by the finite element method.
The special feature here is that the discretization error will be expressed
in the power function of the corresponding eigenvalue. Combined with the upper bounds of the
discretization error \cite{Babuska2,StrangFix}, we are able to estimate the number of reliable
numerical eigenvalues at the predetermined convergence rate.
For more information and examples, please refer to \cite{Zhang}.
Here we provide the corresponding theoretical analysis.

Here we are concerned with the following Laplace eigenvalue problem:

Find $(\lambda, u)$ such that $\|u\|_0=1$ and
\begin{equation}\label{Laplace_Eigenvalue}
\left\{
\begin{array}{rcl}
-\Delta u&=&\lambda u\ \ \ {\rm in}\ \Omega,\\
u&=&0\ \ \ \ \ {\rm on}\ \partial\Omega.
\end{array}
\right.
\end{equation}
Based on the partition $\mathcal{T}_h$ on $\bar{\Omega}$, we can define a suitable
finite element space $V_h$ (conforming or nonconforming for the second order elliptic problem)
as in (\ref{FEM_Space_General}) with piecewise polynomials of degree less than $r$.
We solve the eigenvalue problem (\ref{Laplace_Eigenvalue})
by the finite element method. The finite element approximation
$(\lambda_h,u_h)\in\mathcal{R}\times V_h$  of
(\ref{Laplace_Eigenvalue}) can be defined as follows:

Find $(\lambda_h,u_h)\in \mathcal{R}\times V_h$ such that $\|u_h\|_0=1$ and
\begin{eqnarray}\label{Poisson_Eigenvalue_FEM}
a_h(u_h,v_h)&=&\lambda_h(u_h,v_h)\ \ \ \forall v_h\in V_h.
\end{eqnarray}
First, it is well known that the following upper bound of the
discretization error (see \cite{BrennerScott,Ciarlet, StrangFix}) holds
\begin{eqnarray}\label{upper_bound}
\|u-u_h\|_{\ell,p,h} &\leq& Ch^{s-\ell}\lambda^{(s-\ell)/2},\ \ \ \ 0\leq\ell\leq 1,\ \ 0<s\leq r,
\end{eqnarray}
where  $ 1<p<\infty$ and the constant $C$ is independent of the mesh size $h$ and the
eigenvalue $\lambda$.
In the rest of this section, we will prove $Ch^{r-\ell}\lambda^{(r-\ell)/2}$ is also a
lower bound of discretization error $\|u-u_h\|_{\ell,p,h}$ under some suitable conditions.
\begin{lemma}\label{Lambda_Power_Lemma}
For the eigenvalue problem (\ref{Laplace_Eigenvalue}), the following property holds
\begin{eqnarray}\label{Estimate_H_k}
\|\nabla^{\lceil\frac{r}{2}\rceil-\lfloor\frac{r}{2}\rfloor}\Delta^{\lfloor\frac{r}{2}\rfloor}u\|_{0,p,G}
 = \lambda^{\lfloor\frac{r}{2}\rfloor}\|\nabla^{\lceil\frac{r}{2}\rceil-\lfloor\frac{r}{2}\rfloor}u\|_{0,p,G},
\end{eqnarray}
where $G\subset\subset \Omega$,
$\lceil r\rceil$ denotes the smallest integer not less than $r$ and
$\lfloor r\rfloor$ is the biggest integer not larger than $r$.
\end{lemma}
\begin{proof}
From the eigenvalue problem (\ref{Laplace_Eigenvalue}), we have the following equation
\begin{eqnarray*}
\nabla^{\lceil\frac{r}{2}\rceil-\lfloor\frac{r}{2}\rfloor}\Delta^{\lfloor\frac{r}{2}\rfloor}u
= \lambda^{\lfloor\frac{r}{2}\rfloor}\nabla^{\lceil\frac{r}{2}\rceil-\lfloor\frac{r}{2}\rfloor}u{,}
\end{eqnarray*}
which leads to the desired result (\ref{Estimate_H_k}) and the proof is complete.
\end{proof}

From Lemma \ref{Lambda_Power_Lemma}, we state the following
lower bound results of the discretization error.
\begin{theorem}\label{Lower_bound_Eigenvalue_Corollary}
 Assume that
 $\nabla^{\lceil\frac{r}{2}\rceil-\lfloor\frac{r}{2}\rfloor}\Delta^{\lfloor\frac{r}{2}\rfloor}v_h=0$
 for all $K\in\mathcal{T}_h${, $v_h\in V_h$, where the partition $\mathcal{T}_h$
  is quasi-uniform and shape regular.
 If} there is a domain $G\subset\subset \Omega$ such that
 $\|u\|_{0,p,G}>0$, then for any given exact eigenpair $(\lambda,u)$,
  its corresponding eigenpair approximation
$(\lambda_h, u_h)\in \mathcal{R}\times V_h$ in  (\ref{Poisson_Eigenvalue_FEM}) satisfies the following
lower bound of the discretization error
\begin{equation}\label{lower_convergence_2_1}
\|u-u_h\|_{j,p,h}\geq Ch^{r-j}\lambda^{\frac{r}{2}},\ \ \ \ 0\leq j\leq r,
\end{equation}
where the mesh size $h$ is small enough,
$2\leq p\leq \infty$ and $C$ is a positive constant independent of
 $\lambda$ and the mesh size $h$.

Furthermore, if the family $\{\mathcal{T}_h\}$ of partitions  is only shape regular,
then for any given exact eigenpair $(\lambda,u)$, its corresponding approximation $(\lambda_h, u_h)$
has the following lower bounds of the discretization error
\begin{equation}\label{lower_convergence_2_2}
\left(\sum_{K\in\mathcal{T}_h^G}h_K^{p(j-r)}
\big\|u-u_h\big\|_{j,p,K}^p\right)^{\frac{1}{p}} \geq C\lambda^{\frac{r}{2}},\ \ \ \ 0\leq j\leq r
\end{equation}
and
\begin{equation}\label{lower_convergence_2_3}
\left(\sum_{K\in\mathcal{T}_h^G}h_K^{p\big((j-r)+n(\frac{1}{p}-\frac{1}{q})\big)}
\big\|u-u_h\big\|_{j,q,K}^p\right)^{\frac{1}{p}} \geq C\lambda^{\frac{r}{2}},\ \ \ \ 0\leq j\leq r,
\end{equation}
where the mesh size $h$ is also small enough, $2\leq p <\infty$, $p\leq q\leq \infty$,  $C$ are positive
constants independent of $\lambda$ and the mesh size $h$.
\end{theorem}
\begin{proof}
First, it is easy to see that the eigenfunctions of problem (\ref{Laplace_Eigenvalue}) cannot
 be polynomials of bounded degree on any subdomain $G\subset\subset\Omega$. It means that we also have
$\|\nabla u\|_{0,p,G}>0.$

Now, we present the derivation for the two cases of the positive integer $r$. Toward this end, let $\Pi_h^r$ denote
a suitable piecewise $\mathcal{P}_r$ interpolation operator (e.g., piecewise $L^2$-projection).

In the first case, $r$ is even. Then we have
\begin{eqnarray*}
\|\Delta^{\frac{r}{2}}u\|_{0,p,G}&=&\|\Delta^{\frac{r}{2}}(u-v_h)\|_{0,p,G,h}\nonumber\\
&\leq& \|\Delta^{\frac{r}{2}}(u-\Pi_h^ru)\|_{0,p,G,h}+\|\Delta^{\frac{r}{2}}(\Pi_h^ru-v_h)\|_{0,p,G,h}\nonumber\\
&\leq&Ch^{\delta}\|u\|_{r+\delta,p,G}+Ch^{j-r}\|\Pi_h^ru-v_h\|_{j,p,G,h}\nonumber\\
&\leq&Ch^{\delta}\|u\|_{r+\delta,p,G}+Ch^{j-r}\|\Pi_h^ru-u\|_{j,p,G,h}\nonumber\\
&&\ \ \ +Ch^{j-r}\|u-v_h\|_{j,p,G,h}\nonumber\\
&\leq& Ch^{\delta}\|u\|_{r+\delta,p,G}+Ch^{j-r}\|u-v_h\|_{j,p,G,h}.
\end{eqnarray*}
It means
\begin{eqnarray}\label{Full_Lower_Bound}
Ch^{j-r}\|u-v_h\|_{j,p,G} &\geq& C\|\Delta^{\frac{r}{2}}u\|_{0,p,G}- Ch^{\delta}\|u\|_{r+\delta,p,G}.
\end{eqnarray}
Together with Lemma \ref{Lambda_Power_Lemma}, when $h$ is small enough, we have
\begin{eqnarray}\label{Lower_Estimate_1}
Ch^{j-r}\|u-v_h\|_{j,p,G,h} &\geq& C\|\Delta^{\frac{r}{2}}u\|_{0,p,G} =C\lambda^{\frac{r}{2}}\|u\|_{0,p,G}
\geq C\lambda^{\frac{r}{2}}.
\end{eqnarray}

In the second case, the integer $r$ is odd. Then we have
\begin{eqnarray*}
\|\nabla\Delta^{\frac{r-1}{2}}u\|_{0,p,G}&=&\|\nabla\Delta^{\frac{r-1}{2}}(u-v_h)\|_{0,p,G}\nonumber\\
&\leq& \|\nabla\Delta^{\frac{r-1}{2}}(u-\Pi_h^ru)\|_{0,p,G}
+\|\nabla\Delta^{\frac{r-1}{2}}(\Pi_h^ru-v_h)\|_{0,p,G}\nonumber\\
&\leq&Ch^{\delta}\|u\|_{r+\delta,p,G}+Ch^{j-r}\|\Pi_h^ru-v_h\|_{j,p,G}\nonumber\\
&\leq&Ch^{\delta}\|u\|_{r+\delta,p,G}+Ch^{j-r}\|\Pi_h^ru-u\|_{j,p,G}\nonumber\\
&&\ \ \ +Ch^{j-r}\|u-v_h\|_{j,p,G}\nonumber\\
&\leq& Ch^{\delta}\|u\|_{r+\delta,p,G}+Ch^{j-r}\|u-v_h\|_{j,p,G}
\end{eqnarray*}
which means
\begin{eqnarray*}
Ch^{j-r}\|u-v_h\|_{j,p,G} &\geq& C\|\nabla\Delta^{\frac{r-1}{2}}u\|_{0,p,G}- Ch^{\delta}\|u\|_{r+\delta,p,G}.
\end{eqnarray*}
Similarly, together with Lemma \ref{Lambda_Power_Lemma}, when $h$ is small enough, the following
estimate holds
\begin{eqnarray}\label{Lower_Estimate_2}
Ch^{j-r}\|u-v_h\|_{j,p,G} &\geq& C\|\nabla\Delta^{\frac{r-1}{2}}u\|_{0,p,G}
=C\lambda^{\frac{r-1}{2}}\|\nabla u\|_{0,p,G}\nonumber\\
&\geq& C_{|\Omega|}\lambda^{\frac{r-1}{2}}\|\nabla u\|_{0,G} \geq C\lambda^{\frac{r}{2}},
\end{eqnarray}
where we used $\|\nabla u\|_{0,\Omega}=\lambda^{1/2}$ and the constant
 $C$ depends on $\|\nabla u\|_{0,G}/\|\nabla u\|_{0,\Omega}$.

Combining (\ref{Lower_Estimate_1}), (\ref{Lower_Estimate_2}) and the  arbitrariness of $v_h$ leads to
\begin{eqnarray*}
\inf_{v_h\in V_h} \frac{\|u-v_h\|_{j,p,G,h}}{h^{r-j}} &\geq & C\lambda^{\frac{r}{2}}.
\end{eqnarray*}
Together with the following property
\begin{eqnarray*}
\frac{\|u-u_h\|_{j,p,h}}{h^{r-j}}&\geq& \frac{\|u-u_h\|_{j,p,G,h}}{h^{r-j}}
\geq \inf_{v_h\in V_h} \frac{\|u-v_h\|_{j,p,G,h}}{h^{r-j}},
\end{eqnarray*}
we obtain the desired result (\ref{lower_convergence_2_1}).
 The error bounds (\ref{lower_convergence_2_2}) and
(\ref{lower_convergence_2_3}) can be proved similarly.
\end{proof}
\begin{remark}
The interior regularity result $u\in W^{r+\delta,p}(G)$ for a subdomain
$G\subset\subset\Omega$ and $\delta>0$
 for the Laplace eigenvalue problem (\ref{Laplace_Eigenvalue}) can be obtained from
\cite[Theorem 8.10 and P.214 with some recursive discussion]{GilbargTrudinger}
 when the right-hand side $f$ is sufficiently smooth.
\end{remark}

Now, we present some conforming and nonconforming elements {whose discretization errors attain lower bounds
 in the forms of (\ref{lower_convergence_2_1}),
(\ref{lower_convergence_2_2}),} and (\ref{lower_convergence_2_3})
 with the help of  Theorem \ref{Lower_bound_Eigenvalue_Corollary}.

First we can obtain the lower bound results for the standard Lagrange type elements
\begin{eqnarray}\label{Lagrange_FEM}
V_h&=&\Big\{v_h|_K\in \mathcal{P}_{\ell}(K)\ {\rm or}\ \mathcal{Q}_1(K),\ \ \ \forall
K\in\mathcal{T}_h\Big\}.
\end{eqnarray}
From  Theorem \ref{Lower_bound_Eigenvalue_Corollary}, the lower bound
results in this section hold with $r=\ell+1$ for $\mathcal{P}_{\ell}(K)$ case and
$r=2$ for $\mathcal{Q}_1(K)$ case.

Then it is also easy to verify the lower bound results for
nonconforming elements Crouzeix-Raviart (CR) and $Q_1$ rotation ($Q_1^{\rm rot}$):
\begin{itemize}
\item
The CR element space, proposed by Crouzeix and Raviart
\cite{CrouzeixRaviart}, is defined on simplicial partitions by
\begin{eqnarray*}
V_{h}&=&\Big\{v\in L^2(\Omega):\ v|_{K} \in \mathcal{P}_1(K), \nonumber\\
 &&\ \ \ \ \int_{F}v|_{K_{1}}ds =\int_{F}v|_{K_{2}}ds\ {\rm if}\ K_{1}\cap K_{2}=F\in \mathcal{E}_h\Big\}.
\end{eqnarray*}
The lower bound result holds with $r=2$ and $\gamma\in {\rm Ind}_2$.
\item
The $Q_{1}^{\rm rot}$ element space, proposed by Rannacher and Turek
\cite{RannacherTurek}, and Arbogast and Chen \cite{ArbogastChen}, is defined on $n$-dimensional block
partitions by
\begin{eqnarray*}
V_{h}&=&\Big\{v\in L^{2}(\Omega):v|_{K} \in Q_{\rm Rot}(K),\nonumber\\
&&\ \ \ \int_{F}v|_{K_{1}}ds =\int_{F}v|_{K_{2}}ds~~{\rm if}\
K_{1}\cap K_{2}=F\in\mathcal{E}_h\Big\},
\end{eqnarray*}
where $Q_{\rm Rot}(K)=\mathcal{P}_1(K)+{\rm span}\big\{x_i^2-x_{i+1}^2\ |\ 1\leq i\leq n-1\big\}$.
The lower bound result holds with $r=2$ and $\gamma$ with $\gamma_i=1, \gamma_j=1$, $1\leq i< j\leq n$.
\end{itemize}
Comparing (\ref{upper_bound}) and (\ref{lower_convergence_2_1}), we see that
all lower and upper bounds of the above examples are sharp if the solution is smooth enough such that
the upper bounds hold.

Now, we investigate the number of reliable eigenvalue approximations
obtained by the finite element method via a rigorous proof for the
procedure outlined in \cite{Zhang}.
The techniques for the analysis are the lower bound result (\ref{lower_convergence_2_1})
and the asymptotic behavior for the $j$-th eigenvalue of the Laplace operator.
In the following analysis, we need the following estimate which comes from \cite[Chapter 6, Section 6.3]{StrangFix}:
\begin{eqnarray}\label{Upper_Bound_Differential}
\|u\|_{r+1,2}\leq C\|\Delta^{\frac{r+1}{2}}u\|_{0,2}\leq C\lambda^{\frac{r+1}{2}}.
\end{eqnarray}
\begin{lemma}\label{Lower_Bound_Eigenvalue_Lemma}
Assume that the mesh $\mathcal{T}_h$ is quasi-uniform and we use the conforming finite element method.
Under the condition of Theorem \ref{Lower_bound_Eigenvalue_Corollary},
 we have the following lower bound for the numerical eigenvalue relative error
\begin{eqnarray}\label{Lower_Bound_Eigenvalue}
\frac{\lambda_h-\lambda}{\lambda}&\geq&C(\lambda h^2)^{r-1} -C(\lambda h^2)^r.
\end{eqnarray}
\end{lemma}
\begin{proof}
When $r$ is even, from (\ref{Full_Lower_Bound}) and (\ref{Upper_Bound_Differential}), we have
\begin{eqnarray}\label{Lower_Bound_u_v_h}
\|u-v_h\|_{1,2,G}&\geq& Ch^{r-1}\|\Delta^{\frac{r}{2}}u\|_{0,2,G} -Ch^r\|u\|_{r+1,2,G}\nonumber\\
&\geq& Ch^{r-1}\|\Delta^{\frac{r}{2}}u\|_{0,2,G} -Ch^r\lambda^{\frac{r+1}{2}}\nonumber\\
&\geq&Ch^{r-1}\lambda^{\frac{r}{2}} -Ch^r\lambda^{\frac{r+1}{2}},\ \ \ \ \forall v_h\in V_h.
\end{eqnarray}
The desired result (\ref{Lower_Bound_Eigenvalue}) can be obtained by combing (\ref{lower_convergence_2_1})
and the following result \cite{Babuska2}
\begin{eqnarray*}
c\|u-u_h\|_1^2\leq \lambda_h-\lambda\leq C\|u-u_h\|_1^2.
\end{eqnarray*}
When $r$ is odd, we can also deduce the desired result (\ref{Lower_Bound_Eigenvalue}) in the similar way.
\end{proof}

In \cite{Weyl}, Weyl gives the following asymptotic behavior for the $j$-th
eigenvalue of the Laplace operator:
\begin{eqnarray}\label{Weyl_Result}
\lambda_j&\approx& 4\pi^2\left(\frac{j}{\omega_n|\Omega|}\right)^{\frac{2}{n}},
\end{eqnarray}
where $\omega_n=\pi^{n/2}/\Gamma(1+n/2)$ denotes the volume of the unit ball in $\mathcal{R}^n$.

Based on the above preparation, we come to prove the results in \cite{Zhang}.
\begin{theorem}\label{Reliable_Number_Laplace_Theorem}
Let $N$ denote the total degrees of freedom in the finite element space $V_h$.
Assume the conditions of Lemma \ref{Lower_Bound_Eigenvalue_Lemma} are satisfied and
$j_{\theta}=N^{\theta}$ for $\theta\in [0,1)$. Then the relative error of the
$j_{\theta}$-th numerical eigenvalue
has the following lower bound
\begin{eqnarray}\label{Error_j_th_Eigenvalue}
\frac{\lambda_{j_{\theta},h}-\lambda_{j_{\theta}}}{\lambda_{j_{\theta}}}&\geq& C(r-1)^{2\theta(r-1)}h^{2(r-1)(1-\theta)}
\end{eqnarray}
and the absolute error of the
$j_{\theta}$-th numerical eigenvalue
has the following lower bound
\begin{eqnarray}\label{Error_j_th_Eigenvalue_Absolute}
\lambda_{j_{\theta},h}-\lambda_{j_{\theta}}&\geq& C(r-1)^{2\theta r}
h^{2r(1-\theta)-2},
\end{eqnarray}
when $N$ is large enough.

On the contrary, the number $j_{\theta}$ of the ``reliable"
numerical eigenvalues with relative error of $\lambda_{j_{\theta}}$ converging
at the rate $h^{2\theta(r-1)}$ for $\theta\in (0,1]$, has the following estimate
\begin{eqnarray}\label{Numerber_Eigenvalue_Laplace}
j_{\theta} &\leq&C(r-1)^{(1-\theta)n}N^{1-\theta}.
\end{eqnarray}
Furthermore, the number $j_{\theta}$ of the ``reliable"
numerical eigenvalues with absolute error of $\lambda_{j_{\theta}}$ converging
at the rate $h^{\theta (2(r-1))}$ for $\theta\in (0,1]$, has the following estimate
\begin{eqnarray}\label{Numerber_Eigenvalue_Laplace_Absolute}
j_{\theta} &\leq& C(r-1)^{\frac{r-1}{r}(1-\theta)n}N^{\frac{(r-1)(1-\theta)}{r}}.
\end{eqnarray}
\end{theorem}
\begin{proof}
First, the piecewise polynomial space of degree $r-1$ has the total degree of freedom
$N=\mathcal{O}((r-1)^nh^{-n})$ which means
\begin{eqnarray}\label{Condition_4}
h&=&\mathcal{O}\big(N^{-\frac{1}{n}}(r-1)^{-1}\big).
\end{eqnarray}
Then the desired results (\ref{Error_j_th_Eigenvalue}) and (\ref{Error_j_th_Eigenvalue_Absolute})
can be derived by combining
 (\ref{Lower_Bound_Eigenvalue}), (\ref{Weyl_Result}), (\ref{Condition_4}) and the following estimate
 \begin{eqnarray}\label{Less_Than_1}
 j_{\theta}^{\frac{2}{n}}h^2 \leq C(r-1)^{-2}N^{\frac{2}{n}(\theta-1)}\ll 1,\ \ \  {\rm when}\ \theta<1\
 {\rm and}\ N\ {\rm is\ large\ enough}.
 \end{eqnarray}

By using the lower bound results (\ref{Lower_Bound_Eigenvalue}) and (\ref{Weyl_Result}) and
along the way in \cite{Zhang}, we can give the proof for (\ref{Numerber_Eigenvalue_Laplace}).
 From (\ref{Lower_Bound_Eigenvalue}) and (\ref{Weyl_Result}), we have
\begin{eqnarray*}
\frac{\lambda_{j,h}-\lambda_j}{\lambda_j}&\geq& C\big(j^{\frac{2}{n}}h^2\big)^{r-1}-C\big(j^{\frac{2}{n}}h^2\big)^{r}.
\end{eqnarray*}
Then the desired convergence order $h^{\theta(2r-1)}$ leads to the following estimate
\begin{eqnarray}\label{Condition_5}
h^{\theta(2(r-1))}&\geq& C\big(j^{\frac{2}{n}}h^2\big)^{r-1}(1-j^{\frac{2}{n}}h^2).
\end{eqnarray}
Combining (\ref{Condition_4}),  (\ref{Less_Than_1}) and (\ref{Condition_5}) leads to the desired number $j_{\theta}$
satisfying the inequality (\ref{Numerber_Eigenvalue_Laplace}) and (\ref{Numerber_Eigenvalue_Laplace_Absolute})
can be proved similarly.
\end{proof}


\section{Lower bounds for finite element approximation of $2m$-th order elliptic eigenvalue problem}

We consider similar lower bounds of the discretization error for the $2m$-th order elliptic eigenvalue
 problem by the finite element method. This is a natural generalization of the results in Section
 \ref{Lower_Bound_Elliptic_Eigenvalue_Section}.

 The $2m$-th order Dirichlet elliptic eigenvalue problem for a given integer $m\geq 1$ is defined as
\begin{equation}\label{2m_Problem_Eigenvalue}
\left\{
\begin{array}{rcl}
(-1)^m\Delta^{m}u&=&\lambda u\ \ \ {\rm in}\ \Omega,\\
\frac{\partial^ju}{\partial^j\mathbf \nu}&=&0\ \ \ \ \ {\rm on}\ \partial\Omega\ {\rm and}\ 0\leq j\leq m-1,\\
\|u\|_0&=&1,
\end{array}
\right.
\end{equation}
where $\nu$ denotes the unit outer normal.
The corresponding weak form of problem (\ref{2m_Problem_Eigenvalue}) is defined as follows:

 Find $(\lambda,u)\in\mathcal{R}\times H_0^m(\Omega)$ such that $\|u\|_0=1$ and
\begin{eqnarray}\label{2m_Eigenvalue}
a(u,v)&=&\lambda (u,v)\ \ \ \ \forall v\in H_0^m(\Omega),
\end{eqnarray}
where
$$a(u,v)=\int_{\Omega}\nabla^mu\nabla^mv\;d\Omega.$$

Based on the partition $\mathcal{T}_h$ of $\bar{\Omega}$, we build a suitable
finite element space $V_h$ (conforming or nonconforming for the $2m$-th order elliptic problem)
 with a piecewise polynomial of a degree less than $r$ { and
define the corresponding discrete eigenvalue problem in the finite element space:}

Find $(\lambda_h,u_h)\in\mathcal{R}\times V_h$ such that $\|u_h\|_0=1$ and
\begin{eqnarray}\label{2m_Eigenvalue_FEM}
a_h(u_h,v_h)&=&\lambda_h (u_h,v_h)\ \ \ \ \forall v_h\in V_h,
\end{eqnarray}
where $$a_h(u_h,v_h)=\sum_{K\in\mathcal{T}_h}\int_{K}\nabla^mu_h\nabla^m v_hdK.$$

{
Now we conduct the corresponding lower-bound analysis for the eigenvalue problem (\ref{2m_Eigenvalue}).}
Similarly to Lemma \ref{Lambda_Power_Lemma}, the following estimate for the eigenfunction
 norm by the eigenvalue holds.
\begin{lemma}\label{Lambda_Power_Lemma_2m}
For the eigenvalue problem (\ref{2m_Problem_Eigenvalue}), the following property holds
\begin{eqnarray}\label{Estimate_2m_eigenvalue}
\|\nabla^{mi}\Delta^{m\ell}u\|_{0,p,G}= \lambda^{\ell}\|\nabla^{mi}u\|_{0,p,G},
\end{eqnarray}
where $G\subset\subset \Omega$, $i=0,1$ and $\ell=0,1,\cdots$.
\end{lemma}

From Lemma \ref{Lambda_Power_Lemma_2m}, we state the following
lower bound results of the discretization error.
\begin{theorem}\label{Lower_bound_Eigenvalue_Corollary_2m}
Let $r=mi+2m\ell$.
 Assume that $\nabla^{mi}\Delta^{m\ell}v_h=0$ for $i=0,1$ and $\ell=0,1,\cdots$
 for all $K\in\mathcal{T}_h${, $v_h\in V_h$,
 where the partition $\mathcal{T}_h$ is quasi-uniform and shape regular.
 If} there is a domain $G\subset\subset \Omega$ such that
 $\|u\|_{0,p,G}>0${, then} for any given exact eigenpair
 $(\lambda,u)$, its corresponding eigenpair approximation
$(\lambda_h, u_h)\in \mathcal{R}\times V_h$ in  (\ref{2m_Eigenvalue_FEM}) satisfies the following
lower bound of the discretization error
\begin{equation}\label{lower_convergence_2_1_2m}
\|u-u_h\|_{j,p,h}\geq Ch^{r-j}\lambda^{\frac{r}{2m}},\ \ \ \ 0\leq j\leq r,
\end{equation}
where the mesh size $h$ is small enough, $2\leq p\leq \infty$ and $C$ is a positive constant
independent of $u$, $\lambda$ and the mesh size $h$.

Furthermore, if the family $\{\mathcal{T}_h\}$ of partitions  is
  only shape regular, then for any given exact eigenpair $(\lambda,u)$, its corresponding
 approximation $(\lambda_h, u_h)$
has the following lower bounds of the discretization error
\begin{equation}\label{lower_convergence_2_2_2m}
\left(\sum_{K\in\mathcal{T}_h^G}h_K^{p(j-r)}
\big\|u-u_h\big\|_{j,p,K}^p\right)^{\frac{1}{p}} \geq C\lambda^{\frac{r}{2m}},\ \ \ \ 0\leq j\leq r
\end{equation}
and
\begin{equation}\label{lower_convergence_2_3_2m}
\left(\sum_{K\in\mathcal{T}_h^G}h_K^{p\big((j-r)+n(\frac{1}{p}-\frac{1}{q})\big)}
\big\|u-u_h\big\|_{j,q,K}^p\right)^{\frac{1}{p}} \geq C\lambda^{\frac{r}{2m}},\ \ \ \ 0\leq j\leq r.
\end{equation}
where the mesh size $h$ is small enough, $2\leq p <\infty$, $p\leq q\leq \infty$,  $C$ are positive
constants independent of $u$, $\lambda$ and the mesh size $h$.
\end{theorem}
\begin{proof}
First, it is also easy to { see} that the eigenfunctions of problem (\ref{2m_Problem_Eigenvalue})
cannot be polynomials of bounded degree on any subdomain $G\subset\subset\Omega$.
It means that we also have $\|\nabla^m u\|_{0,p,G}>0.$

 Now, we present the derivation for the two cases of the integer $i$.
In the first case, $i=0$ and $r=2m\ell$. Then we have
\begin{eqnarray}\label{Inequality_3}
\|\Delta^{m\ell}u\|_{0,p,G}&=&\|\Delta^{m\ell}(u-v_h)\|_{0,p,G,h}\nonumber\\
&\leq& \|\Delta^{m\ell}(u-\Pi_h^ru)\|_{0,p,G,h}+\|\Delta^{m\ell}(\Pi_h^ru-v_h)\|_{0,p,G,h}\nonumber\\
&\leq&Ch^{\delta}\|u\|_{r+\delta,p,G}+Ch^{j-r}\|\Pi_h^ru-v_h\|_{j,p,G,h}\nonumber\\
&\leq&Ch^{\delta}\|u\|_{r+\delta,p,G}+Ch^{j-r}\|\Pi_h^ru-u\|_{j,p,G,h}\nonumber\\
&&\ \ \ +Ch^{j-r}\|u-v_h\|_{j,p,G,h}\nonumber\\
&\leq& Ch^{\delta}\|u\|_{r+\delta,p,G}+Ch^{j-r}\|u-v_h\|_{j,p,G,h},
\end{eqnarray}
where $\Pi_h^r$ is defined as in the proof of Theorem \ref{Lower_bound_Eigenvalue_Corollary}.
The inequality (\ref{Inequality_3}) means the following inequality holds
\begin{eqnarray}\label{Lower_Bound_Full_2m}
Ch^{j-r}\|u-v_h\|_{j,p,G,h} &\geq& C\|\Delta^{m\ell}u\|_{0,p,G}- Ch^{\delta}\|u\|_{r+\delta,p,G}.
\end{eqnarray}
Together with Lemma \ref{Lambda_Power_Lemma_2m}, when $h$ is small enough, we have
\begin{eqnarray}\label{Lower_Estimate_1_2m}
Ch^{j-r}\|u-v_h\|_{j,p,G,h} &\geq& C\|\Delta^{m\ell}u\|_{0,p,G} =C\lambda^{\ell}\|u\|_{0,p,G}
\geq C\lambda^{\frac{r}{2m}}.
\end{eqnarray}

In the second case, the integer $i=1$ and $r=m+2m\ell$. Then we have
\begin{eqnarray*}
\|\nabla^{m}\Delta^{m\ell}u\|_{0,p,G}&=&\|\nabla^{m}\Delta^{m\ell}(u-v_h)\|_{0,p,G,h}\nonumber\\
&\leq& \|\nabla^{m}\Delta^{m\ell}(u-\Pi_h^ru)\|_{0,p,G,h}
+\|\nabla^{m}\Delta^{m\ell}(\Pi_h^ru-v_h)\|_{0,p,G,h}\nonumber\\
&\leq&Ch^{\delta}\|u\|_{r+\delta,p,G}+Ch^{j-r}\|\Pi_h^ru-v_h\|_{j,p,G,h}\nonumber\\
&\leq&Ch^{\delta}\|u\|_{r+\delta,p,G}+Ch^{j-r}\|\Pi_h^ru-u\|_{j,p,G,h}\nonumber\\
&&\ \ \ +Ch^{j-r}\|u-v_h\|_{j,p,G,h}\nonumber\\
&\leq& Ch^{\delta}\|u\|_{r+\delta,p,G}+Ch^{j-r}\|u-v_h\|_{j,p,G,h}
\end{eqnarray*}
which means
\begin{eqnarray*}
Ch^{j-r}\|u-v_h\|_{j,p,G,h} &\geq& C\|\nabla^{m}\Delta^{m\ell}u\|_{0,p,G}- Ch^{\delta}\|u\|_{r+\delta,p,G}.
\end{eqnarray*}
Similarly, together with Lemma \ref{Lambda_Power_Lemma_2m}, when $h$ is small enough, we have
\begin{eqnarray}\label{Lower_Estimate_2_2m}
Ch^{j-r}\|u-v_h\|_{j,p,G,h} &\geq& C\|\nabla^{m}\Delta^{m\ell}u\|_{0,p,G}
=C\lambda^{\ell}\|\nabla^m u\|_{0,p,G}\nonumber\\
&\geq&C_{|G|}\lambda^{\ell}\|\nabla^m u\|_{0,G} \geq C\lambda^{\ell+\frac{1}{2}}
=C\lambda^{\frac{r}{2m}},
\end{eqnarray}
where the constant $C$ depends on $\|\nabla^m u\|_{0,G}/\|\nabla^m u\|_{0,\Omega}$ and
 the equality $\|\nabla^m u\|_{0,\Omega}=\lambda^{1/2}$ is used.

Combining (\ref{Lower_Estimate_1_2m}), (\ref{Lower_Estimate_2_2m}) and the  arbitrariness of $v_h$ leads to
\begin{eqnarray*}
\inf_{v_h\in V_h} \frac{\|u-v_h\|_{j,p,G,h}}{h^{r-j}} &\geq & C\lambda^{\frac{r}{2m}}.
\end{eqnarray*}
Together with the following property
\begin{eqnarray*}
\frac{\|u-u_h\|_{j,p,h}}{h^{r-j}}&\geq& \frac{\|u-u_h\|_{j,p,G,h}}{h^{r-j}}
\geq \inf_{v_h\in V_h} \frac{\|u-v_h\|_{j,p,G,h}}{h^{r-j}},
\end{eqnarray*}
we { obtain} the desired result (\ref{lower_convergence_2_1_2m}).
 { The error bounds} (\ref{lower_convergence_2_2_2m}) and
(\ref{lower_convergence_2_3_2m}) can be proved similarly.
\end{proof}

\begin{remark}
The interior regularity result $u\in W^{r+\delta,p}(G)$ for a subdomain $G\subset\subset\Omega$
and $\delta>0$ for problem (\ref{2m_Problem_Eigenvalue}) can be obtained from
\cite[Theorem 7.1.2]{Grisvard} with $f$ replaced by $\lambda u$.
\end{remark}
Now, we discuss some conforming and nonconforming elements that can produce the lower bound of
the discretization error with the help of Theorem \ref{Lower_bound_Eigenvalue_Corollary_2m}.
For the two-dimensional case ($n=2$) there exist elements such as
the Argyris and Hsieh--Clough--Tocher elements \cite{Ciarlet} that yield lower-bound results from
Theorem \ref{Lower_bound_Eigenvalue_Corollary_2m} for the biharmonic problem. The lower-bound results
in this section hold for the Argyris element with $m=2$, $r=6$ ($i=1, \ell=1$) and
 ${\rm Ind}_{6,{\rm used}}=\{\alpha:\ |\alpha|=6\}${; and
for} the Hsieh--Clough--Tocher element with $m=2$, $r=4$ ($i=0, \ell=1$),
${\rm Ind}_{6,{\rm used}}=\{\alpha:\ |\alpha|=4\}$.

Similarly, we can investigate the number of reliable numerical eigenvalues
obtained from the finite element method  for the biharmonic eigenvalue problem (in (\ref{2m_Eigenvalue})
with $m=2$).
The ingredients for the analysis are the lower bound result (\ref{Lower_Bound_Full_2m})
and the asymptotic behavior for the $j$-th eigenvalue of the biharmonic operator.

In \cite{Pleijel}, Pleijel proves the following asymptotic behavior for the $n$-th
eigenvalue of the biharmonic operator:
\begin{eqnarray}\label{Pleijel_Result}
\lambda_j&\approx& 16\pi^4\left(\frac{j}{\omega_n|\Omega|}\right)^{\frac{4}{n}}.
\end{eqnarray}
For simplicity, we only state the following result and the proof can be given in the similar way
of Theorem \ref{Reliable_Number_Laplace_Theorem}.
\begin{theorem}\label{Reliable_Number_Biharmonic_Theorem}
Let $N$ denote the total degrees of freedom in the finite element space $V_h$.
Assume the mesh $\mathcal{T}_h$ is quasi-uniform and we use the conforming finite element method to solve
the biharmonic eigenvalue problem ((\ref{2m_Eigenvalue}) with $m=2$) and set
$j_{\theta}=N^{\theta}$ for $\theta\in [0,1)$. Then the relative error of the $j_{\theta}$-th numerical eigenvalue
has the following lower bound
\begin{eqnarray}\label{Error_j_th_Eigenvalue_Biharmonic}
\frac{\lambda_{j_{\theta},h}-\lambda_{j_{\theta}}}{\lambda_{j_{\theta}}}&\geq& C(r-1)^{2\theta(r-2)}h^{2(r-2)(1-\theta)}
\end{eqnarray}
and the absolute error of the $j_{\theta}$-th numerical eigenvalue
has the following lower bound
\begin{eqnarray}\label{Error_j_th_Eigenvalue_Biharmonic_Absolute}
\lambda_{j_{\theta},h}-\lambda_{j_{\theta}}&\geq& C(r-1)^{2\theta r}h^{2r(1-\theta)-4}.
\end{eqnarray}
On the contrary, the number $j_{\theta}$ of the ``reliable"
numerical eigenvalues with relative error of $\lambda_{j_{\theta}}$ converging
at the rate $h^{\theta (2(r-1))}$ for $\theta\in (0,1]$, has the following estimate
\begin{eqnarray}\label{Numerber_Eigenvalue_Biharmonic}
j_{\theta} &\leq&C(r-1)^{(1-\theta)n}N^{1-\theta}.
\end{eqnarray}
Furthermore,
the number $j_{\theta}$ of the ``reliable"
numerical eigenvalues with absolute error of $\lambda_{j_{\theta}}$ converging
at the rate $h^{\theta (2(r-1))}$ for $\theta\in (0,1]$, has the following estimate
\begin{eqnarray}\label{Numerber_Eigenvalue_Biharmonic_Absolute}
j_{\theta} &\leq&C(r-1)^{\frac{n(1-\theta)(r-2)}{r}}N^{\frac{(1-\theta)(r-2)}{r}}.
\end{eqnarray}
\end{theorem}

\section{Concluding remarks}

In this paper, we have established some new upper and lower error bounds of the finite element approximations.
As an important application, we derive lower bounds of the discretization error
for the Laplacian eigenvalue problem by piecewise polynomial approximation. We also show the asymptotic
convergence behavior for the large numerical eigenvalue approximations.
Our results reveals that the traditional upper bound
is also a lower bound when we solve the eigenvalue problem by the finite element method.
In particular, their dependence on the eigenvalue have the same power for both lower and upper bounds.
This fact should have some illumination
to the design of numerical methods for eigenvalue problems.


\section*{Acknowledgments}
The first author is supported in part by the National Natural Science Foundation
of China (NSFC 91330202,  11371026, 11001259, 11201501, 2011CB309703)
and the National Center for Mathematics and Interdisciplinary Science,
CAS and the President Foundation of AMSS-CAS.
The second author is supported in part by the  US National Science Foundation
through grant DMS-1115530, DMS-1419040, and the National Natural Science Foundation
 of China (91430216, 11471031).


\vskip0.5cm


\end{document}